\documentclass[11pt,reqno]{amsart}
\usepackage{amsfonts, amsmath, amssymb, color}
\usepackage{amsthm}
\usepackage{booktabs}
\usepackage{enumitem}
\usepackage{float}
\usepackage{hhline}
\usepackage{soul, color, xcolor}
\usepackage{lipsum}
\usepackage{lscape}
\usepackage{subfig}
\usepackage{textcomp}
\usepackage{tikz}
\usetikzlibrary{decorations.pathmorphing,shapes,arrows,positioning}
\usepackage{xcolor}
\usepackage{young}
\usepackage{youngtab}
\usepackage{ytableau}
\usepackage{rotating}
\usepackage{graphicx}
\usepackage[colorlinks=true, linkcolor=blue, citecolor=magenta, urlcolor=blue, backref=page]{hyperref}
\usepackage{scalefnt}
\usetikzlibrary{calc}
\usepackage{a4wide}

\hypersetup{colorlinks,linkcolor=blue,urlcolor=cyan,citecolor=magenta}

\allowdisplaybreaks[4]
\theoremstyle{plain}
\newtheorem{thm}{Theorem}[section]
\newtheorem{lem}[thm]{Lemma}
\newtheorem{prop}[thm]{Proposition}
\newtheorem{cor}[thm]{Corollary}
\newtheorem{rem}[thm]{Remark}

\theoremstyle{definition}

\numberwithin{equation}{section} \errorcontextlines=0

\DeclareMathOperator{\im}{Im}
\DeclareMathOperator{\sgn}{sgn}

\def\D{\mathcal{D}}
\def\E{\mathcal{E}}

\def\E{\mathcal{E}}

\def\H{\mathcal{H}}
\def\c{\overrightarrow{\mathbf{c}}}
\def\r{\overrightarrow{\mathbf{r}}}
\def\M{\mathcal{M}}
\def\A{\mathcal{A}}
\def\N{\mathcal{N}}

\makeatletter

\newcommand{\Rmnum}[1]{\expandafter\@slowromancap\romannumeral #1@}
\makeatother
\newcommand{\la}{\lambda}

\begin{document}
\title{Plethystic Murnaghan--Nakayama rule via vertex operators}
\author{Yue Cao}
\address{School of Mathematics, South China University of Technology, Guangzhou, Guangdong 510640, China}
\email{434406296@qq.com}
\author{Naihuan Jing}
\address{Department of Mathematics, North Carolina State University, Raleigh, NC 27695, USA}
\email{jing@ncsu.edu}
\author{Ning Liu}
\address{Beijing International Center for Mathematical Research, Peking University, Beijing 100871, China}
\email{mathliu123@outlook.com}

\maketitle
\begin{abstract}
  Based on the vertex operator realization of the Schur functions, a determinant-type plethystic Murnaghan--Nakayama rule is obtained and utilized to derive a general formula of the expansion coefficients of $s_{\nu}$ in the plethysm product  $(p_{n}\circ h_{k})s_{\mu}$. Meanwhile, the equivalence between our algebraic rule and the combinatorial one is also established. As an application, we provide a simple way to compute the generalized Waring formula.
\end{abstract}

\tableofcontents

\section{Introduction}
The Murnaghan--Nakayama rule \cite{Mur, Nak1, Nak2} is an iterative combinatorial formula for irreducible character values
 $\chi^{\lambda}_{\mu}$ of the symmetric group in terms of ribbon tableaux. Under the Frobenius characteristic map, it is equivalent to
 a statement on Schur symmetric functions:
 \begin{align}\label{e:MN}
 p_{r}s_{\mu}=\sum_{\lambda}(-1)^{ht(\lambda/\mu)}s_{\lambda},
 \end{align}
where $ p_{r}$ is the $r$th power sum symmetric function, $s_{\lambda}$ is the Schur function labeled by
partition $\lambda$, and the sum runs over all partitions $\lambda$ such that $\lambda/\mu$ is a border strip of size
$r$. Here $ht(\lambda/\mu)$ is the height of the skew diagram $\lambda/\mu$.

The plethystic Murnaghan--Nakayama rule was introduced and proved in \cite{DLT} using Muir's rule. The combinatorial rule states that
\begin{align}\label{t:c}
(p_n\circ h_k)s_{\mu}=\sum_{\lambda}\sgn(\lambda/\mu)s_{\lambda}
\end{align}
summed over all partitions $\lambda$ such that $\lambda/\mu$ is a horizontal $n$-border strip of weight $k$. Here $\sgn(\la/\mu)$ is defined in \eqref{e:sgn}. Clearly it generalizes both
the classical Murnaghan--Nakayama rule $(k=1)$ and the Pieri rule $(n=1)$. An alternative proof was furnished in \cite{EPW} using the character theory of the symmetric group, and a direct combinatorial proof was given in \cite{W}.

Recently there is a renewed interest on the Murnaghan--Nakayama rule for various algebras
such as the Hecke algebra and the Hecke-Clifford algebra (cf. \cite{JL1,JL2}).
The aim of this paper is to derive a determinant-type plethystic Murnaghan--Nakayama rule. Our approach utilizes vertex operator realization of symmetric functions \cite{J1, J2}. We express the coefficient of $s_{\lambda}$ in $(p_n\circ h_k)s_{\mu}$ by determinants, and we obtain an algebraic formulation of Murnaghan--Nakayama rule, which states
\begin{align}\label{e:algMN}
(p_{n}\circ h_{k})s_{\mu}=\sum_{\lambda\vdash |\mu|+nk}\det (\M(\lambda/\mu))s_{\lambda}.
\end{align}

One advantage of \eqref{e:algMN} is that 
the matrix $\M(\lambda/\mu)$ is defined by data of the partitions (see Sec. \ref{s:pMN}, also \eqref{e:m_{ij}}),
thus offers an effectively algebraic method to compute the plethystic
Murnaghan-Nakayama rule. In addition, by defining column and row transformations of the matrix, a combinatorial characterization for the determinant of $\M(\lambda/\mu)$ is given. This combinatorial description is related to the horizontal $n$-border strip of weight $k$, which explicitly reads
\begin{align}\label{e:Comb.Descrip.}
\det(\M(\la/\mu))=
\begin{cases}
\sgn(\la/\mu) & \text{if $\la/\mu$ is a horizontal $n$-border strip of weight $k$}\\
0 & \text{otherwise},
\end{cases}
\end{align}
which shows the equivalence between our algebraic formula and the combinatorial form given in \cite{DLT,W}. The derivation of \eqref{e:Comb.Descrip.} relies crucially on a new characterization of the horizontal $n$-border strip of weight $k$, which we will present in Sec. \ref{s:prelim}.

We remark that a spin counterpart of the plethystic Murnaghan--Nakayama rule is studied in the separate paper \cite{CJL} using the vertex operator realization of the Schur $Q$-functions.

The layout of the paper is as follows. In Sec. \ref{s:prelim}, we review some basic concepts and give a new characterization of a horizontal $n$-border strip of weight $k$ (Prop. \ref{p:characterization}). In Sec. \ref{s:pMN}, the main results will be presented. We firstly introduce the vertex operator $V^{(n)}(z)$ and its adjoint operator $V^{(n)*}(z)$, and then based on the usual techniques of vertex operators, we express the coefficients $a^{\lambda}_{(n,k)\mu}$ of $s_{\lambda}$ in the expansion of $(p_n\circ h_k)s_{\mu}$, which derives the determinant-type plethystic Murnaghan--Nakayama rule (Thm. \ref{t:dpM-N}). This allows us compute $a^{\lambda}_{(n,k)\mu}$ via determinant $\det(\M(\la/\mu))$. Therefore, we can easily deduce the combinatorial version by establishing a combinatorial description for $\det(\M(\la/\mu))$. At the end of this section, a determinant-based method is offered to compute the generalized Waring formula (Cor. \ref{c:Waring}).

\section{Preliminaries}\label{s:prelim}
In this section, we firstly give a new characterization of a horizontal $n$-border strip of weight $k$, and then review some basic notions regarding symmetric functions and plethysm. The standard reference are \cite{Mac,Stanley}.
\subsection{Partitions}
A {\em partition} $\lambda=\left(\lambda_{1},\lambda_{2},\cdots ,\lambda_{t}\right)$ of {\em length} $\ell(\lambda)=t$ is a list of $t$ non-increasing (nonzero) integers.
The {\em weight} of $\lambda$ is $|\lambda|=\sum_{i=1}^{t}\lambda_{i}$, so we usually write $\lambda \vdash|\lambda|$. Let $P_{n} =\{\lambda\in P: |\lambda|=n\}$, where $P$ is the set of partitions. We identify a partition $\lambda$ with its {\em Young diagram} $D(\lambda)=\{(i, j)\in P\times P : 1\leq i\leq \lambda_{i}, 1\leq j\leq t\}$. Sometime we write $\lambda=\left(1^{m_{1}(\lambda)}, 2^{m_{2}(\lambda)},\cdots\right)$, where $m_{i}(\lambda)$ is the number of times that $i$ appears among the parts of $\lambda$. For $\lambda\vdash n$ denote
\begin{align}\label{e:2-1}
z_{\lambda}=\prod_{i\geq  1}i^{m_{i}(\lambda)}m_{i}(\lambda)!
\end{align}
The {\em conjugate} $\lambda^{\prime}$ of a partition $\lambda$ obtained by changing rows to columns in $D(\lambda)$, so its Young diagram $D(\lambda^{\prime})=\{(j, i)\in P\times P :(i, j)\in D(\lambda)\}$.
 A {\em composition} $\alpha=\left(\alpha_{1}, \alpha_{2}, \ldots, \alpha_{l}\right)$ is a list of nonnegative integers, denoted as $\alpha\models|\alpha|=\sum_{i=1}^{l}\alpha_{i}$.

Let $\lambda, \mu\in P$ such that $\lambda \supset \mu$, i.e., $\lambda_i\geq \mu_i$ for all $i$. The skew Young diagram $\lambda / \mu$ is the part of the diagram $\lambda$ obtained by removing that of $\mu$, so $(\lambda/\mu)_i=\lambda_i-\mu_i$.
A skew Young diagram $\lambda / \mu$ is an {\it $n$-border strip} if it consists of $n$ connected boxes and does not contain any $2 \times 2$ square $\boxplus$. Its {\it height} (or {\it spin}) is the number of the rows minus 1.
Geometrically an $n$-border strip is a string of $n$ boxes with the starting box $s$ and the ending box $e$ (see Fig. 1).
\begin{figure}[H]
 \centering
 \scalefont{0.8}
 \begin{tikzpicture}[scale=0.4]
    \coordinate (Origin)   at (0,0);
    \coordinate (XAxisMin) at (0,0);
    \coordinate (XAxisMax) at (9,0);
    \coordinate (YAxisMin) at (-3,0);
    \coordinate (YAxisMax) at (0,-9);
    \draw [thin, black] (-0.75,0) -- (-0.25,0);
     \draw [thin, black] (-0.5,0) -- (-0.5,-6);
    \draw [thin, black] (-0.75,-6) -- (-0.25,-6);
    \draw [thin, black] (4,0) -- (9,0);
    \draw [thin, black] (4,-1) --(9,-1);
    \draw [thin, black] (3,-2) -- (5,-2);
    \draw [thin, black] (1,-3) -- (5,-3);
    \draw [thin, black] (1,-4) -- (4,-4);
    \draw [thin, black] (0,-5) -- (2,-5);
    \draw [thin, black] (0,-6) -- (2,-6);
    \draw [thin, black] (0,-5) -- (0,-6);
     \draw [thin, black] (1,-3) -- (1,-6);
     \draw [thin, black] (2,-3) -- (2,-6);
     \draw [thin, black] (3,-2) -- (3,-4);
     \draw [thin, black] (4,0) -- (4,-4);
      \draw [thin, black] (5,0) -- (5,-3);
     \draw [thin, black] (6,0) -- (6,-1);
     \draw [thin, black] (7,0) -- (7,-1);
     \draw [thin, black] (8,0) -- (8,-1);
     \draw [thin, black] (9,0) -- (9,-1);
     \draw [thin, black] (4.5,-0.5) -- (8.5,-0.5);
     \draw [thin, black] (4.5,-0.5) -- (4.5,-2.5);
     \draw [thin, black] (3.5,-2.5) -- (4.5,-2.5);
     \draw [thin, black] (3.5,-2.5) -- (3.5,-3.5);
     \draw [thin, black] (1.5,-3.5) -- (3.5,-3.5);
     \draw [thin, black] (1.5,-3.5) -- (1.5,-5.5);
     \draw [thin, black] (0.5,-5.5) -- (1.5,-5.5);
  \node[inner sep=2pt] at (8.5,-0.5) {\textcolor{blue}{\bf$\bullet$}};
  \node[inner sep=2pt] at (8.5,0.5) {$s$};
  \node[inner sep=2pt] at (0.5,-5.5) {$\bullet$};
   \node[inner sep=2pt] at (0.5,-6.5) {$e$};
  \node[inner sep=2pt] at (-1,-3) {$6$};
     \end{tikzpicture}
    \caption{The path of $\lambda/\mu$ with the starting box $s$ and the ending box $e$, and $spin(\lambda/\mu)=5$.}
    \end{figure}
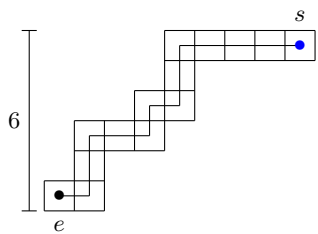

Let $(\lambda/\mu)\uparrow$ be the horizontal strip made of the top boxes of the columns of $\lambda/\mu$ (see Fig. 2).
 \begin{figure}[H]
\begin{tikzpicture}[scale=0.4]
    \coordinate (Origin)   at (0,0);
    \coordinate (XAxisMin) at (0,0);
    \coordinate (XAxisMax) at (9,0);
    \coordinate (YAxisMin) at (0,0);
    \coordinate (YAxisMax) at (0,-9);
    \draw [thin, black] (6,0) -- (11,0);
    \draw [thin, black] (6,-1) --(11,-1);
    \draw [thin, black] (3,-2) -- (7,-2);
    \draw [thin, black] (3,-3) -- (7,-3);
    \draw [thin, black] (0,-4) -- (4,-4);
    \draw [thin, black] (0,-5) -- (4,-5);
    \draw [thin, black] (0,-6) -- (1,-6);
    \draw [thin, black] (0,-4) -- (0,-6);
     \draw [thin, black] (1,-4) -- (1,-6);
     \draw [thin, black] (2,-4) -- (2,-5);
     \draw [thin, black] (3,-2) -- (3,-5);
     \draw [thin, black] (4,-2) -- (4,-5);
      \draw [thin, black] (5,-2) -- (5,-3);
     \draw [thin, black] (6,0) -- (6,-3);
     \draw [thin, black] (7,0) -- (7,-3);
     \draw [thin, black] (8,0) -- (8,-1);
     \draw [thin, black] (9,0) -- (9,-1);
     \draw [thin, black] (10,0) -- (10,-1);
     \draw [thin, black] (11,0) -- (11,-1);
      \node[inner sep=2pt] at (0.5,-4.5) {$\ast$};
     \node[inner sep=2pt] at (1.5,-4.5) {$\ast$};
     \node[inner sep=2pt] at (2.5,-4.5) {$\ast$};
     \node[inner sep=2pt] at (3.5,-2.5) {$\ast$};
     \node[inner sep=2pt] at (4.5,-2.5) {$\ast$};
     \node[inner sep=2pt] at (5.5,-2.5) {$\ast$};
    \node[inner sep=2pt] at (6.5,-0.5) {$\ast$};
    \node[inner sep=2pt] at (7.5,-0.5) {$\ast$};
     \node[inner sep=2pt] at (8.5,-0.5) {$\ast$};
     \node[inner sep=2pt] at (9.5,-0.5) {$\ast$};
     \node[inner sep=2pt] at (10.5,-0.5) {$\ast$};
     \end{tikzpicture}
    \caption{A horizontal strip $(\lambda/\mu)\uparrow$ composed of all $\ast$'s locations.}
    \end{figure}
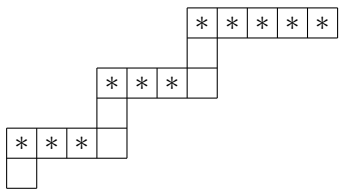

\subsection{$n$-border strip decomposition}
For two partitions $\mu\subset\la$, a chain decomposition of $\la/\mu$
\begin{align}\label{e:chain}
\underline{\alpha}: \mu=\alpha^{(0)}\subset\alpha^{(1)}\subset\cdots\subset\alpha^{(l)}=\lambda
\end{align}
 is called an {\it $n$-border strip chain} of weight $\tau=(\tau_1,\tau_2,\cdots,\tau_l)$ $(\tau_i\geq1)$, if each $\alpha^{(i)}/\alpha^{(i-1)}$ is an $\tau_in$-border strip for $i=1,\ldots,l$. In this case, we say the length of the chain $\underline{\alpha}$ is $l$, denote by $l(\underline{\alpha})$ and also let
\begin{align}
\sgn(\underline{\alpha}):=\prod_{i=1}^{l}(-1)^{ht(\alpha^{(i)}/\alpha^{(i-1)})}.
\end{align}

Let $\underline{\alpha}$ be an $n$-border strip chain of weight $\tau=(\tau_1,\tau_2,\cdots,\tau_l)$ $(\tau_i\geq1)$. We introduce two conditions:
\begin{enumerate}[label= {\bf(\roman*)}]
	\item $\tau_i=1$ for all $i=1,\ldots,l$. That is, each $\alpha^{(i)}/\alpha^{(i-1)}$ is an $n$-boder strip.
	\item each starting box $s^i$ of $\alpha^{(i)}/\alpha^{(i-1)}~(i=1,\ldots,l)$ lies in $(\lambda/\mu)\uparrow$ and the starting box of $\alpha^{(i)}/\alpha^{(i-1)}$ lies in the southwest of that of $\alpha^{(j)}/\alpha^{(j-1)}$ provided that $i<j$.
\end{enumerate}


The following notions are essential in the subsequent context:
\begin{itemize}

\item If $\underline{\alpha}$ satisfies {\bf (ii)}, then $\underline{\alpha}$ is called a horizontal $n$-border strip chain of weight $\tau$ (see Fig. 4);

\item If $\underline{\alpha}$ satisfies both {\bf (i)} and {\bf (ii)}, then $\underline{\alpha}$ is called a horizontal $n$-border strip chain of weight $l$.
\end{itemize}
For an $n$-border strip chain of weight $\tau$, we usually omit the weight $\tau$ for simplification unless otherwise specified. Fig. 3 gives an example for a horizontal $4$-border strip chain of weight $5$.
 \begin{figure}[H]
\centerline{\begin{tikzpicture}[scale=0.5]
    \coordinate (Origin)   at (0,0);
    \coordinate (XAxisMin) at (0,0);
    \coordinate (XAxisMax) at (9,0);
    \coordinate (YAxisMin) at (0,0);
    \coordinate (YAxisMax) at (0,-9);
    \draw [thin, black] (4,0) -- (8,0);
    \draw [thin, black] (4,-1) --(8,-1);
    \draw [thin, black] (3,-2) -- (8,-2);
    \draw [thin, black] (2,-3) -- (8,-3);
    \draw [thin, black] (0,-4) -- (4,-4);
    \draw [thin, black] (0,-5) -- (3,-5);
    \draw [thin, black] (0,-6) -- (1,-6);
    \draw [thin, black] (0,-7) -- (1,-7);
     \draw [thin, black] (1,-4) -- (1,-7);
     \draw [thin, black] (2,-3) -- (2,-5);
     \draw [thin, black] (3,-2) -- (3,-5);
     \draw [thin, black] (4,0) -- (4,-4);
      \draw [thin, black] (5,0) -- (5,-3);
     \draw [thin, black] (6,0) -- (6,-3);
     \draw [thin, black] (7,0) -- (7,-3);
     \draw [thin, black] (8,0) -- (8,-3);
      \draw [thin, black] (0,-4) -- (0,-7);
      \draw [thin, black] (7.5,-0.5) -- (7.5,-2.5);
      \draw [thin, black] (6.5,-2.5) -- (7.5,-2.5);
      \draw [thin, black] (6.5,-0.5) -- (6.5,-1.5);
       \draw [thin, black] (5.5,-1.5) -- (6.5,-1.5);
       \draw [thin, black] (5.5,-1.5) -- (5.5,-2.5);
       \draw [thin, black] (4.5,-0.5) -- (5.5,-0.5);
       \draw [thin, black] (4.5,-0.5) -- (4.5,-2.5);
       \draw [thin, black] (3.5,-2.5) -- (3.5,-3.5);
       \draw [thin, black] (2.5,-3.5) -- (3.5,-3.5);
       \draw [thin, black] (2.5,-3.5) -- (2.5,-4.5);
       \draw [thin, black] (0.5,-4.5) -- (1.5,-4.5);
       \draw [thin, black] (0.5,-4.5) -- (0.5,-6.5);
       \node[inner sep=2pt] at (0.5,-6.5) {$\bullet$};
       \node[inner sep=2pt] at (0.5,-7.5) {$e^{1}$};
     \node[inner sep=2pt] at (1.5,-4.5) {\textcolor{blue}{$\bullet$}};
       \node[inner sep=2pt] at (1.5,-3.5) {$s^{1}$};
     \node[inner sep=2pt] at (2.5,-4.5) {$\bullet$};
      \node[inner sep=2pt] at (2.5,-5.5) {$e^{2}$};
     \node[inner sep=2pt] at (3.5,-2.5)  {\textcolor{blue}{$\bullet$}};
     \node[inner sep=2pt] at (3.5,-1.5) {$s^{2}$};
     \node[inner sep=2pt] at (4.5,-2.5) {$\bullet$};
      \node[inner sep=2pt] at (4.5,-3.5) {$e^{3}$};
     \node[inner sep=2pt] at (5.5,-0.5)  {\textcolor{blue}{$\bullet$}};
     \node[inner sep=2pt] at (5.5,0.5) {$s^{3}$};
     \node[inner sep=2pt] at (5.5,-2.5) {$\bullet$};
     \node[inner sep=2pt] at (5.5,-3.5) {$e^{4}$};
     \node[inner sep=2pt] at (6.5,-0.5)  {\textcolor{blue}{$\bullet$}};
      \node[inner sep=2pt] at (6.5,0.5) {$s^{4}$};
     \node[inner sep=2pt] at (6.5,-2.5) {$\bullet$};
     \node[inner sep=2pt] at (6.5,-3.5) {$e^{5}$};
      \node[inner sep=2pt] at (7.5,-0.5)  {\textcolor{blue}{$\bullet$}};
       \node[inner sep=2pt] at (7.5,0.5) {$s^{5}$};
     \end{tikzpicture}}
    \caption{A horizontal $4$-border strip chain of weight $5$.}
    \end{figure}
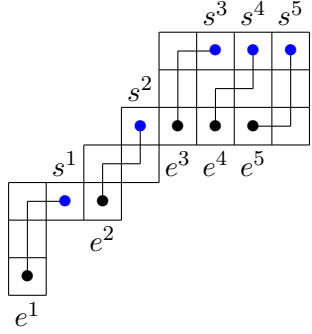
    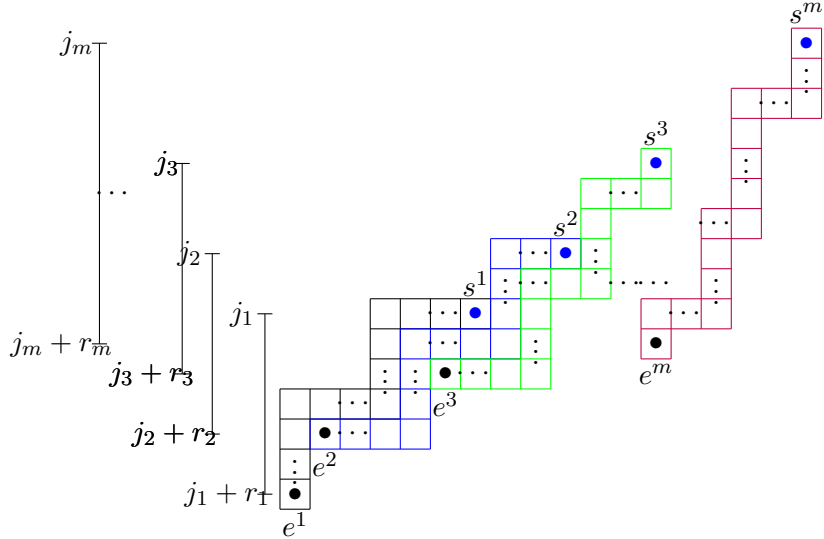
\begin{figure}[H]
\begin{tikzpicture}[scale=0.4]\label{f:H}
    \coordinate (Origin)   at (0,0);
    \coordinate (XAxisMin) at (-10,0);
    \coordinate (XAxisMax) at (20,0);
    \coordinate (YAxisMin) at (0,9);
    \coordinate (YAxisMax) at (0,-9);
    \draw [thin, black] (-0.75,-2.5) -- (-0.25,-2.5);
    \draw [thin, black] (-0.5,-2.5) -- (-0.5,-8.5);
     \draw [thin, black] (-0.75,-8.5) -- (-0.25,-8.5);
     \draw [thin, black] (-2.5,-0.5) -- (-2,-0.5);
     \draw [thin, black] (-2.25,-0.5) -- (-2.25,-6.5);
      \draw [thin, black] (-2.5,-6.5) -- (-2,-6.5);
       \draw [thin, black] (-3.5,2.5) -- (-3,2.5);
     \draw [thin, black] (-3.25,2.5) -- (-3.25,-4.5);
      \draw [thin, black] (-3.5,-4.5) -- (-3,-4.5);

       \draw [thin, black] (-6.25,6.5) -- (-5.75,6.5);
     \draw [thin, black] (-6,6.5) -- (-6,-3.5);
      \draw [thin, black] (-6.25,-3.5) -- (-5.75,-3.5);

   \draw [thin, black] (3,-2) -- (7,-2);
    \draw [thin, black] (3,-3) -- (7,-3);
    \draw [thin, black] (3,-4) -- (4,-4);
    \draw [thin, black] (0,-5) -- (4,-5);
    \draw [thin, black] (0,-6) -- (4,-6);
    \draw [thin, black] (0,-7) -- (1,-7);
     \draw [thin, black] (0,-8) -- (1,-8);
     \draw [thin, black] (0,-9) -- (1,-9);
     \draw [thin, black] (0,-5) -- (0,-9);
     \draw [thin, black] (1,-5) -- (1,-9);
      \draw [thin, black] (2,-5) -- (2,-6);
     \draw [thin, black] (3,-2) -- (3,-6);
     \draw [thin, black] (4,-2) -- (4,-6);
     \draw [thin, black] (5,-2) -- (5,-3);
      \draw [thin, black] (6,-2) -- (6,-3);
      \draw [thin, black] (7,-2) -- (7,-3);
      \draw [thin, blue](7,0) -- (10,0);
       \draw [thin, blue](7,-1) -- (10,-1);
       \draw [thin, blue](7,-2) -- (8,-2);
       \draw [thin, blue](4,-3) -- (8,-3);
       \draw [thin, blue](4,-4) -- (8,-4);
       \draw [thin, blue](4,-5) -- (5,-5);
       \draw [thin, blue](1,-6) -- (5,-6);
       \draw [thin, blue](1,-7) -- (5,-7);
       \draw [thin, blue](1,-6) -- (1,-7);
       \draw [thin, blue](2,-6) -- (2,-7);
       \draw [thin, blue](3,-6) -- (3,-7);
       \draw [thin, blue](4,-3) -- (4,-7);
       \draw [thin, blue](5,-3) -- (5,-7);
       \draw [thin, blue](6,-3) -- (6,-4);
       \draw [thin, blue](7,0) -- (7,-4);
       \draw [thin, blue](8,0) -- (8,-4);
       \draw [thin, blue](9,0) -- (9,-1);
       \draw [thin, blue](10,0) -- (10,-1);
       \draw [thin, green](5,-4) -- (9,-4);
        \draw [thin, green](5,-5) -- (9,-5);
       \draw [thin, green](8,-3) -- (9,-3);
       \draw [thin, green](8,-2) -- (11,-2);
       \draw [thin, green](8,-1) -- (11,-1);
       \draw [thin, green](10,0) -- (11,0);
        \draw [thin, green](10,1) -- (13,1);
        \draw [thin, green](10,2) -- (13,2);
         \draw [thin, green](12,3) -- (13,3);
        \draw [thin, green](13,1) -- (13,3);
        \draw [thin, green](12,3) -- (12,1);
        \draw [thin, green](11,2) -- (11,-2);
         \draw [thin, green](10,2) -- (10,-2);
          \draw [thin, green](9,-1) -- (9,-5);
         \draw [thin, green](8,-1) -- (8,-5);
         \draw [thin, green](7,-4) -- (7,-5);
         \draw [thin, green](6,-4) -- (6,-5);
         \draw [thin, green](5,-4) -- (5,-5);
          \draw [thin,purple](12,-4) -- (13,-4);
          \draw [thin,purple](12,-3) -- (15,-3);
         \draw [thin,purple](12,-2) -- (15,-2);
         \draw [thin,purple](14,-1) -- (15,-1);
         \draw [thin,purple](14,0) -- (16,0);
          \draw [thin,purple](14,1) -- (16,1);
          \draw [thin,purple](15,2) -- (16,2);
           \draw [thin,purple](15,3) -- (16,3);
           \draw [thin,purple](15,4) -- (18,4);
         \draw [thin,purple](15,5) -- (18,5);
         \draw [thin,purple](17,6) -- (18,6);
          \draw [thin,purple](17,7) -- (18,7);
          \draw [thin,purple](18,7) -- (18,4);
          \draw [thin,purple](17,7) -- (17,4);
          \draw [thin,purple](16,5) -- (16,0);
          \draw [thin,purple](15,5) -- (15,-3);
          \draw [thin,purple](14,1) -- (14,-3);
           \draw [thin,purple](13,-2) -- (13,-4);
           \draw [thin,purple](12,-2) -- (12,-4);
          \node[inner sep=2pt] at (6.5,-2.5) {\textcolor{blue}{\bf$\bullet$}};
          \node[inner sep=2pt] at (6.5,-1.5) {$s^{1}$};
         \node[inner sep=2pt] at (0.5,-8.5) {$\bullet$};
         \node[inner sep=2pt] at (0.5,-9.5) {$e^{1}$};
         \node[inner sep=2pt] at (9.5,-0.5){\textcolor{blue}{\bf$\bullet$}};
         \node[inner sep=2pt] at (9.5,0.5) {$s^{2}$};
         \node[inner sep=2pt] at (1.5,-6.5) {$\bullet$};
         \node[inner sep=2pt] at (1.5,-7.5) {$e^{2}$};
         \node[inner sep=2pt] at (12.5,2.5) {\textcolor{blue}{\bf$\bullet$}};
         \node[inner sep=2pt] at (12.5,3.5) {$s^{3}$};
         \node[inner sep=2pt] at (5.5,-4.5) {$\bullet$};
         \node[inner sep=2pt] at (5.5,-5.5) {$e^{3}$};
         \node[inner sep=2pt] at (17.5,6.5) {\textcolor{blue}{\bf$\bullet$}};
         \node[inner sep=2pt] at (17.5,7.5) {$s^{m}$};
         \node[inner sep=2pt] at (12.5,-3.5) {$\bullet$};
          \node[inner sep=2pt] at (12.5,-4.5) {$e^{m}$};
        \node[inner sep=2pt] at (8.5,-0.5) {$\cdots$};
   \node[inner sep=2pt] at (7.5,-1.5) {$\vdots$};
    \node[inner sep=2pt] at (5.5,-2.5) {$\cdots$};
     \node[inner sep=2pt] at (5.5,-3.5) {$\cdots$};
     \node[inner sep=2pt] at (3.5,-4.5) {$\vdots$};
      \node[inner sep=2pt] at (4.5,-4.5) {$\vdots$};
       \node[inner sep=2pt] at (2.5,-5.5) {$\cdots$};
        \node[inner sep=2pt] at (2.5,-6.5) {$\cdots$};
         \node[inner sep=2pt] at (0.5,-7.5) {$\vdots$};
         \node[inner sep=2pt] at (11.5,1.5) {$\cdots$};
         \node[inner sep=2pt] at (10.5,-0.5) {$\vdots$};
         \node[inner sep=2pt] at (8.5,-1.5) {$\cdots$};
         \node[inner sep=2pt] at (8.5,-3.5) {$\vdots$};
         \node[inner sep=2pt] at (6.5,-4.5) {$\cdots$};
         \node[inner sep=2pt] at (11.5,-1.5) {$\cdots$};
          \node[inner sep=2pt] at (12.5,-1.5) {$\cdots$};
         \node[inner sep=2pt] at (17.5,5.5) {$\vdots$};
          \node[inner sep=2pt] at (16.5,4.5) {$\cdots$};
         \node[inner sep=2pt] at (15.5,2.5) {$\vdots$};
          \node[inner sep=2pt] at (14.5,0.5) {$\cdots$};
          \node[inner sep=2pt] at (14.5,-1.5) {$\vdots$};
          \node[inner sep=2pt] at (13.5,-2.5) {$\cdots$};
         \node[inner sep=2pt] at (-3,-0.5) {$j_{2}$};
          \node[inner sep=2pt] at (-1.25,-2.5) {$j_{1}$};
          \node[inner sep=2pt] at (-1.75,-8.5) {$j_{1}+r_{1}$};
          \node[inner sep=2pt] at (-3.5,-6.5) {$j_{2}+r_{2}$};
           \node[inner sep=2pt] at (-3.75,2.5) {$j_{3}$};
          \node[inner sep=2pt] at (-4.25,-4.5) {$j_{3}+r_{3}$};
          \node[inner sep=2pt] at (-3.5,-6.5) {$j_{2}+r_{2}$};
           \node[inner sep=2pt] at (-3.75,2.5) {$j_{3}$};
          \node[inner sep=2pt] at (-4.25,-4.5) {$j_{3}+r_{3}$};
           \node[inner sep=2pt] at (-5.5,1.5) {$\cdots$};
           \node[inner sep=2pt] at (-6.75,6.5) {$j_{m}$};
          \node[inner sep=2pt] at (-7.25,-3.5) {$j_{m}+r_{m}$};
     \end{tikzpicture}
    \caption{A horizontal $n$-border strip chain in $\lambda/\mu$.}
    \end{figure}
\begin{rem}\label{R:horizontal}
For given partitions $\mu\subset\la$, a horizontal $n$-border strip chain of weight $l$ is unique if it exists. This can be proven by induction on $l$. Therefore, in this case, we also say $\la/\mu$ is a horizontal $n$-border strip of weight $l$ and define
\begin{align}\label{e:sgn}
\sgn(\la/\mu):=\sgn(\underline{\alpha})=\prod_{i=1}^{l}(-1)^{ht(\alpha^{(i)}/\alpha^{(i-1)})}.
\end{align}
It is worth pointing that a horizontal $n$-border strip chain is not usually unique for a given skew diagram.

\end{rem}

In a skew diagram $\E$, we number the rows (resp. columns) of $\E$ from top to bottom (resp. from left to right). A box has {\em coordinate} $(i,j)$ if it occupies the $i$th row and the $j$th column and we say that the box $(i,j)$ has {\em content $j-i$}. Clearly, two boxes have the same content if and only if they are on the same diagonal.
Let $\D(\E)$ be the set of all $n$-border strip decompositions of $\E$.

Now we introduce a map $\omega:\D(\E)\rightarrow \D(\E)$ by the following procedure.
 \begin{enumerate}[label=Step \arabic*:]
    \item For a given decomposition $\underline{\alpha}$ of $\E$, find the border strip (denoted by $\Theta$) whose starting box is the northeasternmost box of $\E$;

    \item If there is no border strip below $\Theta$, then we view $\underline{\alpha}\setminus\Theta$ and $\E\setminus\Theta$ as $\underline{\alpha}$ and $\E$ respectively. We finish this procedure if $\underline{\alpha}$ (or $\E$) is $\varnothing$ and return to Step 1 if else;

    \item If there is a border strip (denoted by $\Phi$) below $\Theta$, then
     \begin{itemize}
\item If the starting box of $\Phi$ is directly below the ending box of $\Theta$ (see Fig. 5), then we view $\Theta\dot\cup\Phi$ as $\Theta$ and return to Step 2;

\item If the starting box of $\Phi$ is not directly below the ending box of $\Theta$, then exchange the boxes with the same content in $\Theta\dot\cup\Phi$ (see Fig. 6). We view the resulting border strip $\Theta^{'}$ as $\Theta$ and return to Step 2.
\end{itemize}
\end{enumerate}
%

This procedure clearly terminates after finitely many steps. 
It is clear that the resulting $n$-border strip decomposition is horizontal. That is, $\omega$ converts an $n$-border strip chain of $\E$ into a horizontal one.

\begin{figure}[H]
\centerline{\begin{tikzpicture}[scale=0.4]
    \coordinate (Origin)   at (0,0);
    \coordinate (XAxisMin) at (-3,0);
    \coordinate (XAxisMax) at (11,0);
    \coordinate (YAxisMin) at (0,0);
    \coordinate (YAxisMax) at (0,-9);
  \draw [thin, black] (6,0) -- (9,0);
  \draw [thin, black] (6,-1) -- (9,-1);
    \draw [thin, black] (4,-2) --(7,-2);
    \draw [thin, black] (4,-3) -- (7,-3);
    \draw [thin, black] (4,-4) -- (5,-4);
    \draw [thin, black] (4,-5) -- (5,-5);
    \draw [thin, black] (4,-2) -- (4,-5);
    \draw [thin, black] (5,-2) -- (5,-5);
     \draw [thin, black] (6,0) -- (6,-3);
     \draw [thin, black] (7,0) -- (7,-3);
      \draw [thin, black] (8,0) -- (8,-1);
      \draw [thin, black] (9,0) -- (9,-1);
    \draw [thin, blue](2,-5) -- (5,-5);
       \draw [thin, blue](2,-6) -- (5,-6);
       \draw [thin, blue](0,-7) -- (3,-7);
       \draw [thin, blue](0,-8) -- (3,-8);
       \draw [thin, blue](0,-7) -- (0,-8);
       \draw [thin, blue](1,-7) -- (1,-8);
       \draw [thin, blue](2,-5) -- (2,-8);
       \draw [thin, blue](3,-5) -- (3,-8);
       \draw [thin, blue](4,-5) -- (4,-6);
       \draw [thin, blue](5,-5) -- (5,-6);
      \node[inner sep=2pt] at (8.5,-0.5) {\textcolor{blue}{\bf$\bullet$}};
      \node[inner sep=2pt] at (9.5,-0.5) {$s^{1}$};
      \node[inner sep=2pt] at (10.5,-0.5) {$\Theta$};
         \node[inner sep=2pt] at (4.5,-4.5) {$\bullet$};
          \node[inner sep=2pt] at (5.5,-4.5) {$e^{1}$};
         \node[inner sep=2pt] at (4.5,-5.5) {\textcolor{blue}{\bf$\bullet$}};
          \node[inner sep=2pt] at (5.5,-5.5) {$s^{2}$};
         \node[inner sep=2pt] at (0.5,-7.5) {$\bullet$};
         \node[inner sep=2pt] at (-0.5,-7.5) {$e^{2}$};
         \node[inner sep=2pt] at (-1.5,-7.5) {$\Phi$};
        \node[inner sep=2pt] at (7.5,-0.5) {$\cdots$};
   \node[inner sep=2pt] at (6.5,-1.5) {$\vdots$};
   \node[inner sep=2pt] at (5.5,-2.5) {$\cdots$};
     \node[inner sep=2pt] at (4.5,-3.5) {$\vdots$};
    \node[inner sep=2pt] at (3.5,-5.5) {$\cdots$};
         \node[inner sep=2pt] at (2.5,-6.5) {$\vdots$};
         \node[inner sep=2pt] at (1.5,-7.5) {$\cdots$};
     \end{tikzpicture}}
    \caption{ $s^2$ is directly below $e^1$.}
    \end{figure}
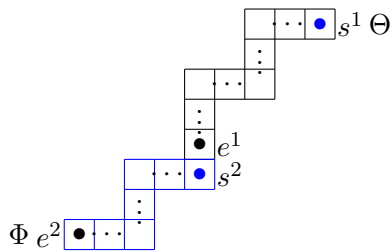
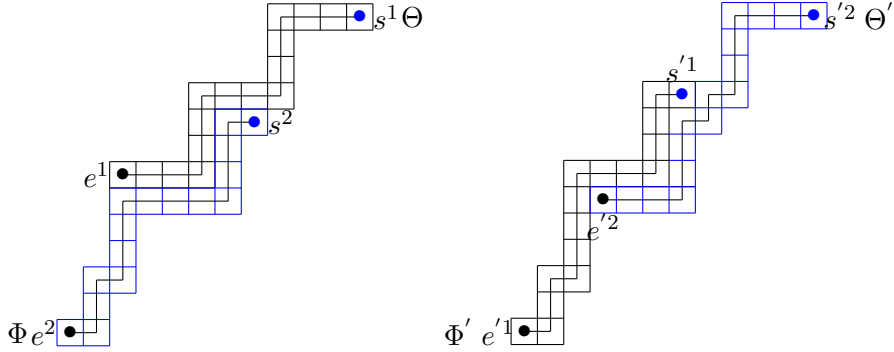
\begin{figure}[H]
\begin{tikzpicture}[scale=0.35]
    \coordinate (Origin)   at (0,0);
    \coordinate (XAxisMin) at (-3,0);
    \coordinate (XAxisMax) at (13,0);
    \coordinate (YAxisMin) at (0,0);
    \coordinate (YAxisMax) at (0,-14);
   \draw [thin, black] (8,0) -- (12,0);
    \draw [thin, black] (8,-1) -- (12,-1);
    \draw [thin, black] (8,-2) -- (9,-2);
    \draw [thin, black] (5,-3) -- (9,-3);
    \draw [thin, black] (5,-4) -- (9,-4);
    \draw [thin, black] (5,-5) -- (6,-5);
     \draw [thin, black] (2,-6) -- (6,-6);
     \draw [thin, black] (2,-7) -- (6,-7);
     \draw [thin, black] (2,-6) -- (2,-7);
     \draw [thin, black] (3,-6) -- (3,-7);
      \draw [thin, black] (4,-6) -- (4,-7);
     \draw [thin, black] (5,-3) -- (5,-7);
     \draw [thin, black] (6,-3) -- (6,-7);
     \draw [thin, black] (7,-3) -- (7,-4);
      \draw [thin, black] (8,0) -- (8,-4);
      \draw [thin, black] (9,0) -- (9,-4);
      \draw [thin, black] (10,0) -- (10,-1);
      \draw [thin, black] (11,0) -- (11,-1);
      \draw [thin, black] (12,0) -- (12,-1);
      \draw [thin, blue](6,-4) -- (8,-4);
       \draw [thin, blue](6,-5) -- (8,-5);
       \draw [thin, blue](6,-6) -- (7,-6);
       \draw [thin, blue](2,-7) -- (7,-7);
       \draw [thin, blue](2,-8) -- (7,-8);
       \draw [thin, blue](2,-9) -- (3,-9);
       \draw [thin, blue](1,-10) -- (3,-10);
       \draw [thin, blue](1,-11) -- (3,-11);
       \draw [thin, blue](0,-12) -- (2,-12);
       \draw [thin, blue](0,-13) -- (2,-13);
       \draw [thin, blue](0,-12) -- (0,-13);
       \draw [thin, blue](1,-10) -- (1,-13);
       \draw [thin, blue](2,-7) -- (2,-13);
       \draw [thin, blue](3,-7) -- (3,-11);
       \draw [thin, blue](4,-7) -- (4,-8);
       \draw [thin, blue](5,-7) -- (5,-8);
       \draw [thin, blue](6,-4) -- (6,-8);
       \draw [thin, blue](7,-4) -- (7,-8);
        \draw [thin, blue](8,-4) -- (8,-5);
        \draw [thin, black] (8.5,-0.5) -- (11.5,-0.5);
       \draw [thin, black] (8.5,-0.5) -- (8.5,-3.5);
       \draw [thin, black] (5.5,-3.5) -- (8.5,-3.5);
       \draw [thin, black] (5.5,-3.5) -- (5.5,-6.5);
       \draw [thin, black] (2.5,-6.5) -- (5.5,-6.5);
       \draw [thin, black] (6.5,-4.5) -- (7.5,-4.5);
       \draw [thin, black] (6.5,-4.5) -- (6.5,-7.5);
       \draw [thin, black] (2.5,-7.5) -- (6.5,-7.5);
       \draw [thin, black] (2.5,-7.5) -- (2.5,-10.5);
       \draw [thin, black] (1.5,-10.5) -- (1.5,-12.5);
       \draw [thin, black] (1.5,-12.5) -- (0.5,-12.5);
       \draw [thin, black] (1.5,-10.5) -- (2.5,-10.5);
        \node[inner sep=2pt] at (11.5,-0.5) {\textcolor{blue}{\bf$\bullet$}};
        \node[inner sep=2pt] at (12.5,-0.5) {$s^{1}$};
        \node[inner sep=2pt] at (13.5,-0.5) {$\Theta$};
         \node[inner sep=2pt] at (2.5,-6.5) {$\bullet$};
          \node[inner sep=2pt] at (1.5,-6.5) {$e^{1}$};
         \node[inner sep=2pt] at (7.5,-4.5) {\textcolor{blue}{\bf$\bullet$}};
         \node[inner sep=2pt] at (8.5,-4.5) {$ s^{2}$};
         \node[inner sep=2pt] at (0.5,-12.5) {$\bullet$};
        \node[inner sep=2pt] at (-0.5,-12.5) {$e^{2}$};
        \node[inner sep=2pt] at (-1.5,-12.5) {$\Phi$};
     \end{tikzpicture}
     ~~
   \begin{tikzpicture}[scale=0.35]
     \coordinate (Origin)   at (0,0);
    \coordinate (XAxisMin) at (-3,0);
    \coordinate (XAxisMax) at (13,0);
    \coordinate (YAxisMin) at (0,0);
    \coordinate (YAxisMax) at (0,-14);
   \draw [thin, black] (5,-3) -- (7,-3);
    \draw [thin, black] (5,-4) -- (7,-4);
    \draw [thin, black] (5,-5) -- (6,-5);
     \draw [thin, black] (2,-6) -- (6,-6);
     \draw [thin, black] (2,-7) -- (7,-7);
     \draw [thin, black] (2,-8) -- (3,-8);
     \draw [thin, black] (2,-9) -- (3,-9);
      \draw [thin, black] (1,-10) -- (3,-10);
     \draw [thin, black] (1,-11) -- (3,-11);
     \draw [thin, black] (0,-12) -- (2,-12);
     \draw [thin, black] (0,-13) -- (2,-13);
      \draw [thin, black] (0,-12) -- (0,-13);
      \draw [thin, black] (1,-10) -- (1,-13);
      \draw [thin, black] (2,-6) -- (2,-13);
      \draw [thin, black] (3,-6) -- (3,-11);
      \draw [thin, black] (4,-6) -- (4,-7);
        \draw [thin, black] (5,-3) -- (5,-7);
        \draw [thin, black] (6,-3) -- (6,-7);
        \draw [thin, black] (7,-3) -- (7,-4);
      \draw [thin, blue](8,0) -- (12,0);
       \draw [thin, blue](8,-1) -- (12,-1);
       \draw [thin, blue](8,-2) -- (9,-2);
       \draw [thin, blue](7,-3) -- (9,-3);
       \draw [thin, blue](7,-4) -- (9,-4);
       \draw [thin, blue](6,-5) -- (8,-5);
       \draw [thin, blue](6,-6) -- (7,-6);
       \draw [thin, blue](3,-7) -- (7,-7);
      \draw [thin, blue](3,-8) -- (7,-8);
       \draw [thin, blue](3,-7) -- (3,-8);
       \draw [thin, blue](4,-7) -- (4,-8);
       \draw [thin, blue](5,-7) -- (5,-8);
       \draw [thin, blue](6,-7) -- (6,-8);
        \draw [thin, blue](7,-3) -- (7,-8);
        \draw [thin, blue](8,0) -- (8,-5);
       \draw [thin, blue](9,0) -- (9,-4);
       \draw [thin, blue](10,0) -- (10,-1);
       \draw [thin, blue](11,0) -- (11,-1);
       \draw [thin, blue](12,0) -- (12,-1);
        \draw [thin, black] (8.5,-0.5) -- (11.5,-0.5);
       \draw [thin, black] (8.5,-0.5) -- (8.5,-3.5);
       \draw [thin, black] (7.5,-3.5) -- (8.5,-3.5);
       \draw [thin, black] (7.5,-3.5) -- (7.5,-4.5);
       \draw [thin, black] (7.5,-4.5) -- (6.5,-4.5);
       \draw [thin, black] (6.5,-4.5) -- (6.5,-7.5);
       \draw [thin, black] (6.5,-7.5) -- (3.5,-7.5);
       \draw [thin, black] (6.5,-3.5) -- (5.5,-3.5);
       \draw [thin, black] (5.5,-3.5) -- (5.5,-6.5);
       \draw [thin, black] (5.5,-6.5) -- (2.5,-6.5);
       \draw [thin, black] (2.5,-6.5) -- (2.5,-10.5);
       \draw [thin, black] (1.5,-10.5) -- (2.5,-10.5);
       \draw [thin, black] (1.5,-10.5) -- (1.5,-12.5);
        \draw [thin, black] (0.5,-12.5) -- (1.5,-12.5);
       \node[inner sep=2pt] at (6.5,-3.5) {\textcolor{blue}{\bf$\bullet$}};
        \node[inner sep=2pt] at (6.5,-2.5) {$s^{'1}$};
         \node[inner sep=2pt] at (0.5,-12.5) {$\bullet$};
          \node[inner sep=2pt] at (-0.5,-12.5) {$e^{'1}$};
         \node[inner sep=2pt] at (11.5,-0.5) {\textcolor{blue}{\bf$\bullet$}};
         \node[inner sep=2pt] at (12.5,-0.5) {$ s^{'2}$};
         \node[inner sep=2pt] at (14,-0.5) {$\Theta^{'}$};
         \node[inner sep=2pt] at (3.5,-7.5) {$\bullet$};
        \node[inner sep=2pt] at (3.5,-8.5) {$e^{'2}$};
        \node[inner sep=2pt] at (-2,-12.5) {$\Phi^{'}$};
     \end{tikzpicture}
    \caption{Exchange boxes with the same content.}
  \end{figure}
%

Denote by $\H(\E)$ the set of all horizontal $n$-border strip chains of $\E$. Namely, $\H(\E):=\{\underline{\alpha}\in \D(\E)\mid \underline{\alpha}~ \text{is horizontal}\}$. Therefore $\im(\omega)\subset\H(\E)$. We also define $m(\E):=\max\{l(\underline{\alpha})\mid \underline{\alpha}\in \H(\E)\}$, then we have the following characterization of horizontal $n$-border strips.

\begin{prop}\label{p:characterization}
If $\D(\E)\neq\varnothing$ and $|\E|=kn$, then $\E$ is a horizontal $n$-border strip of weight $k$ if and only if $m(\E)=k$.
\end{prop}
\begin{proof} This is shown by definition once noting that $\H(\E)\neq\varnothing$ when $\D(\E)\neq\varnothing$.
\end{proof}

\begin{rem}\label{R:chacter}
If $\D(\la/\mu)\neq\varnothing$, the above characterization tells us we can find a longest horizontal $n$-border strip chain $\underline{\beta}$ of weight $\rho=(\rho_1,\rho_2,\cdots)$ such that
 \begin{enumerate}
\item if $\la/\mu$ is a horizontal $n$-border strip of weight $k$, then $\rho=(1^k)$ and the longest horizontal $n$-border strip chain $\underline{\beta}$ is unique. Moreover, if we denote $\underline{\beta}$ by
    \begin{align*}
    \underline{\beta}: \mu=\beta^{(0)}\subset\beta^{(1)}\subset\cdots\subset\beta^{(k)}=\lambda,
    \end{align*}
then $\beta^{(i)}/\beta^{(i-1)}$ ($1\leq i\leq k$) is an $n$-border strip and the ending box of $\beta^{(i)}/\beta^{(i-1)}$ is in the southwest of that of $\beta^{(j)}/\beta^{(j-1)}$ if and only if $i<j$.
\item if not, then there exists at least one $\rho_i>1$ such that $\beta^{(i)}/\beta^{(i-1)}$ is a disjoint union of $\rho_i$ $n$-border strips and the starting box of the lower border strip is directly below the ending box of the upper border strip (cf. Fig. 5).

\end{enumerate}
\end{rem}

\subsection{Symmetric functions and plethysm}
We recall some notations and definitions about symmetric functions following \cite[Ch.I]{Mac}. Let $\Lambda_{\mathbb{Q}}$ be the ring of symmetric functions in infinitely many variables $x_{1},x_{2},\cdots $ over $\mathbb{Q}$. There are several well-known bases in $\Lambda_{\mathbb{Q}}$ parameterized by partitions. For arbitrary partition $\lambda=(\lambda_{1},\lambda_{2},\cdots \lambda_{t})$, the power sum symmetric functions  $p_{\lambda}=p_{\lambda_{1}}p_{\lambda_{2}}\cdots p_{\lambda_{t}}$,
where $p_{n}=\sum x_{i}^{n}$, form a $\mathbb{Q}$-basis of $\Lambda_{\mathbb{Q}}$. The canonical inner product$\langle\cdot,\cdot\rangle$ is defined by \cite[p.64, (4.7)]{Mac}
\begin{align}\label{e:p}
\langle p_{\lambda},p_{\mu}\rangle=\delta_{\lambda\mu}z_{\lambda}.
\end{align}
where $z_{\lambda}$ is defined in \eqref{e:2-1}. Let $s_{\lambda}$ be the Schur function associated with the partition $\lambda$, then $\{s_{\lambda}|\lambda\in P \}$ form an orthonormal basis of $\Lambda_{\mathbb{Q}}$ under the inner product \cite[p.64, (4.8)]{Mac}
\begin{align}\label{e:s}
\langle s_{\lambda},s_{\mu}\rangle=\delta_{\lambda\mu}.
\end{align}
The monomial symmetric function $m_{\lambda}=\sum_{\sigma}x^{\sigma(\lambda)}$, where $\sigma$ runs through distinct permutations of $\lambda$ as tuples. The complete symmetric functions $h_{\lambda}=h_{\lambda_{1}}h_{\lambda_{2}}\cdots$ and the elementary symmetric functions $e_{\lambda}=e_{\lambda_1}e_{\lambda_2}\cdots $ for $\la=(\la_1,\la_2,\cdots)$ form $\mathbb{Q}$-basis of $\Lambda_{\mathbb{Q}}$ respectively, where $h_{r}=\sum _{\lambda\vdash r}m_{\lambda}$, $e_r=m_{(1^r)}$.

Define the $\mathbb{Q}$-linear and anti-involutive automorphism $\ast$ by
\begin{align}\label{e:fgh}
\langle fg,h\rangle=\langle g,f^{\ast}h\rangle
\end{align}
for any $f,g,h\in \Lambda_{\mathbb{Q}}$. It is clear from \eqref{e:p} that the adjoint $p^{\ast}_{m}=m\frac{\partial}{\partial p_{m}}$ for $m\in\mathbb N$.

An important further operation on symmetric functions is the plethysm \cite[Ch. X.II]{L1}\cite[Ch.I, \S 8]{Mac}.
 For $f\in\Lambda_{\mathbb Q}$, the plethysm $g \mapsto g \circ f$ on any  $g=\sum_{\lambda} c_{\lambda} p_{\lambda}\in\Lambda_{\mathbb Q}$ means
\begin{align}
g\circ f=\sum_{\lambda} c_{\lambda} \prod_{i=1}^{l(\lambda)} f\left(x_{1}^{\lambda_{i}}, x_{2}^{\lambda_{i}}, \ldots\right).
\end{align}
Thus the plethysm by $f$ is an algebra isomorphism of $\Lambda_{\mathbb Q}$ sending $p_{k} \mapsto$ $f\left(x_{1}^{k}, x_{2}^{k}, \ldots\right)$. Since the $p_{\la}$ ($\la\in P$) forms a base of $\Lambda_{\mathbb Q}$, the definition of plethysm is well-defined.

Let $f, g, h$ be symmetric functions. We have the following properties for plethysm \cite{L,L2, Mac}:
\begin{align}
(f+g)\circ h&=f\circ h+g\circ h,\\
(fg)\circ h&=(f\circ h)(g\circ h),\\
f\circ p_k&=p_k \circ f.
\end{align}

\section{Main results}\label{s:pMN}
In this section, a determinant-type plethystic Murnaghan--Nakayama rule is established with the help of the vertex operator realization.

\subsection{Vertex operators}Recall the vertex operator realization of Schur functions \cite{J1}. Define the vertex operator $S(z)$ and its adjoint operator $S^{\ast}(z)$
from $\Lambda_{\mathbb Q}$ to $\Lambda[z, z^{-1}]=\Lambda\otimes \mathbb Q[z, z^{-1}]$ (cf. $t = 0$ in \cite{J2}) by
\begin{align}
S(z) &=\exp \left(\sum_{m=1}^{\infty} \frac{p_{m}}{m} z^{m}\right) \exp \left(-\sum_{m=1}^{\infty} \frac{\partial}{\partial p_{m}} z^{-m}\right)=\sum_{m \in \mathbb{Z}} S_{m} z^{m}, \\
S^{\ast}(z) &=\exp \left(-\sum_{m=1}^{\infty} \frac{p_{m}}{m} z^{m}\right) \exp \left(\sum_{m=1}^{\infty} \frac{\partial}{\partial p_{m}} z^{-m}\right)=\sum_{m \in \mathbb{Z}} S_{m}^{\ast} z^{-m},
\end{align}
where $p_m$ acts by multiplication in $\Lambda_{\mathbb{Q}}$. Note that the components $S_n, S_n^*$ are linear operators on $\Lambda_{\mathbb{Q}}$.
 In this picture, the space $\Lambda_{\mathbb Q}$ can be viewed as a highest weight representation or the Fock space of the infinite-dimensional Heisenberg Lie algebra
 generated by the $p_n$ and its adjoint operator. The vacuum vector or the highest weight vector is the vector $1$ (cf. \cite{J2, J3}).
\begin{prop}
Let $1$ be the {\it vacuum vector} of $\Lambda_\mathbb{Q}$, we have
\begin{enumerate}
\item The components of $S(z)$ and $S^{*}(z)$ satisfy the following commutation relations \cite[Prop. (2.12) for $t = 0$]{J2}:
\begin{align}
S_{m} S_{n}+S_{n-1} S_{m+1}&=0, \\
S_{m}^{*} S_{n}^{*}+S_{n+1}^{*} S_{m-1}^{*}&=0,\\
S_{m} S_{n}^{*}+S_{n-1}^{*} S_{m-1}&=\delta_{m, n}.
\end{align}
\item If $m>0$, then by \cite[Eq. (2.11) for $t = 0$]{J2}, \cite[Eq. (2.12) ]{J3}
\begin{align}\label{e:h_{m}}
h_{m}=S_{m}.1=\sum_{\lambda \vdash m} \frac{p_{\lambda}}{z_{\lambda}}.
\end{align}
\item For any composition $\mu=(\mu_{1},\ldots,\mu_{l})$, then by \cite[Prop. (2.17) for $t = 0$, Prop. (4.3)]{J2}
 $$s_{\mu}=\prod_{i<j}(1-R_{ij})h_{\mu_{1}}h_{\mu_{2}}\cdots h_{\mu_{l}}=S_{\mu_{1}}S_{\mu_{2}}\cdots S_{\mu_{l}}.1.$$
In general, let $\delta=(l-1,l-2,\ldots,1,0),$ then $s_{\mu}=0$ or $\pm s_{\lambda}$ for a partition $\lambda$ such that $\lambda\in \mathfrak{S}_{l}(\mu+\delta)-\delta.$
\end{enumerate}
\end{prop}

\subsection{Plethystic Murnaghan--Nakayama rule via determinants}
For a given positive integer $n$, we introduce the operators $V^{(n)}_{m}$ as follows:
\begin{align}\label{e:V}
V^{(n)}(z)=\exp \left(\sum_{m=1}^{\infty} \frac{p_{n m}}{m} z^{m}\right)=\sum_{m \geq 0} V^{(n)}_{m} z^{m}.
\end{align}

The adjoint operator $V^{(n)*}(z)$ with respect to the canonical inner product \eqref{e:p} can be written as:
\begin{align}\label{e:V^{*}}
V^{(n)*}(z)=\exp \left(\sum_{m=1}^{\infty} n \frac{\partial}{\partial p_{nm}} z^{-m}\right)=\sum_{m \geq 0} V_{m}^{(n)*} z^{-m}.
\end{align}

By \eqref{e:V} and \eqref{e:V^{*}}, we obtain the following results.
\begin{lem} Let $n$ be a positive integer, then
\begin{align}
V^{(n)}_{m}&=\delta_{m, 0} ~~(m\leq 0); \quad V_{m}^{(n)*}.1=\delta_{m, 0} ~~ (m\in \mathbb{Z}),\\ \label{e:plethysm}
V^{(n)}_{m}&=\sum_{\lambda \vdash m} \frac{p_{n \lambda}}{z_{\lambda}}=\sum_{\lambda \vdash m} \frac{p_{\lambda}(x_1^n, x_2^n,\cdots)}{z_{\lambda}}=h_m(x_1^n, x_2^n,\cdots)=p_n\circ h_m, ~~ (m> 0),
\end{align}
where $n\la=(n\la_1,n\la_2,\cdots)$.
\end{lem}

Using \eqref{e:s} and \eqref{e:fgh}, the coefficients $a^{\lambda}_{(n,k)\mu}$ in the plethystic Murnaghan--Nakayama rule can be expressed via vertex operators: 
\begin{align}\label{e:coefficient}
\begin{split}
a^{\lambda}_{(n,k)\mu}=&\langle V^{(n)}_{k}s_{\mu}, s_{\lambda} \rangle\\
=&\langle V^{(n)}_{k}S_{\mu_1}S_{\mu_2}\cdots.1, S_{\lambda_1}S_{\lambda_2}\cdots.1 \rangle\\
=&\langle S_{\mu_1}S_{\mu_2}\cdots.1, V^{(n)*}_{k}S_{\lambda_1}S_{\lambda_2}\cdots.1 \rangle.
\end{split}
\end{align}

The usual vertex operator calculus gives that
\begin{align}
&V^{(n)*}(z) S(w)=S(w) V^{(n)*}(z)\left(1-\frac{w^{n}}{z}\right)^{-1},
\end{align}
which coefficient of $z^{-k} w^{m}$ is
\begin{align}\label{e:Vk*Sm}
V^{(n)*}_{k}S_{m}=\sum_{r=0}^{k}S_{m-nr}V^{(n)*}_{k-r}.
\end{align}

Applying both sides of \eqref{e:Vk*Sm} to the vacuum vector $1$ gives that
\begin{align}
V^{(n)*}_{k}S_{m}.1=S_{m-nk}.1.
\end{align}

Now we give one of the main results of this section.

\begin{thm}\label{t:V*S} Let $\lambda=\left(\lambda_{1}, \cdots, \lambda_{l}\right)\in \mathcal{P}$. For any nonnegative integer $k$, we have that
\begin{align}\label{e:V*S}
V_{k}^{(n)*} S_{\lambda_{1}} \cdots S_{\lambda_{l}}.1=\sum_{\nu \models k} S_{\lambda-n \nu}.1
\end{align}
where $n\nu=(n\nu_1,n\nu_2,\cdots,n\nu_l)$, $S_{\lambda-n \nu}.1=S_{\lambda_1-n \nu_1}S_{\lambda_2-n \nu_2}\cdots S_{\lambda_l-n \nu_l}.1.$
\end{thm}
\begin{proof} This can be obtained by  induction on $l(\lambda)$ using \eqref{e:Vk*Sm}.
\end{proof}

Using the idea of \cite[\S 4]{JL1}, we define an $l(\lambda)\times l(\lambda)$ matrix $\M(\lambda/\mu)=(m_{ij})$, where
\begin{align}\label{e:m_{ij}}
m_{ij}=
\begin{cases}
1 & \text {if $\lambda_{i}-i\geq\mu_{j}-j$, $n|(\lambda_{i}-i-\mu_{j}+j$)},\\
0 & \text {otherwise,}
\end{cases}
\quad 1\leq i,j\leq l(\la).
\end{align}
Note that we add $0's$ at the end of $\mu$ if $l(\mu)<l(\la)$.

Similar to \cite[\S 4]{JL1}, we have
\begin{align}\label{e:V*}
V_{k}^{(n)*} S_{\lambda_{1}} \cdots S_{\lambda_{l}}.1=\sum_{\nu \models k} S_{\lambda-n \nu}.1=\sum_{\lambda\supset\mu\vdash |\lambda|-nk}{\rm det} (\M(\lambda/\mu))S_{\mu}.1.
\end{align}

Combining \eqref{e:coefficient} with \eqref{e:V*} we then have $a^{\lambda}_{(n,k)\mu}={\rm det} (\M(\lambda/\mu))$, which derives the following determinant-type plethystic Murnaghan--Nakayama rule.
\begin{thm} \label{t:algMN}
 For a given $n\in\mathbb{N}_{+}$ and any $k\in\mathbb{N}_{+}$, $\mu\in P$, then
\begin{align}\label{e:pmn}
(p_{n}\circ h_{k})s_{\mu}=\sum_{\mu\subset\lambda\vdash |\mu|+nk}{\rm det} (\M(\lambda/\mu))s_{\lambda}.
\end{align}
\end{thm}
As the matrix $\M(\lambda/\mu)$ is defined by only using data of involved partitions, this version of the Murnaghan-Nakayama rule can
be effectively computed by algebraic method.

\subsection{A combinatorial interpretation for $\det(\M(\la/\mu))$}
A nature question arises whether our determinant-type plethystic Murnaghan--Nakayama rule is equivalent to the combinatorial one. To establish the equivalence, we give a combinatorial interpretation for $\det(\M(\la/\mu))$.

Let $\nu=(\nu_1,\nu_2,\cdots,\nu_r)$ be a composition of $k$. The classical Murnaghan--Nakayama rule for irreducible characters of the symmetric group (see \eqref{e:MN}) tells us
the coefficient of $s_{\la}$ in $p_{n\nu_1}p_{n\nu_2}\cdots p_{n\nu_r}s_{\mu}$ will vanish unless $\la/\mu$ can be decomposed into an $n$-border strip chain of weight $(\nu_r,\cdots,\nu_2,\nu_1)$.

It follows from \eqref{e:plethysm} that
$$(p_n\circ h_k)s_{\mu}=V^{(n)}_{k}s_{\mu}=\sum_{\nu=(\nu_{1},\cdots,\nu_{l})\vdash k}\frac{p_{n\nu_{1}}\cdots p_{n\nu_{l}}s_{\mu}}{z_{\nu}}.$$
Therefore, only those $\la's$ contribute in the summation in \eqref{e:pmn} such that $\la/\mu$ is decomposable into an $n$-border strip chain of some weight $\rho$ ($\rho\models k$). Namely, $\det(\M(\la/\mu))=0$ unless $\D(\la/\mu)\neq\varnothing$ with the notations in Sec. \ref{s:prelim}.

%

Next, we introduce an $l(\lambda)\times l(\lambda)$ control matrix $\A(\lambda/\mu)=(a_{ij})$ of $\M(\lambda/\mu)$, whose entries are defined by
\begin{align}\label{e:a_{ij}}
a_{ij}=\lambda_i-i-\mu_j+j, \quad 1\leq i, j\leq l(\lambda)
\end{align}
Then $m_{ij}=1$ if $n|a_{ij}\geq0$ and $0$ otherwise. Sometimes we adapt the notation $\A(\la/\mu)_{ij}$ to present the $i$th row and $j$th column entry in $\A(\la/\mu)$. As we can observe that the matrix $\A(\lambda/\mu)$ has the following properties:
\begin{enumerate}[label={\bf(\Alph*)}]
\item the entries along each row (or column) are strictly
increasing from left to right (decreasing from top to bottom);

\item the difference between any two entries in each column (resp. rows) is fixed and independent of the rows (resp. columns).
\end{enumerate}

Now we introduce two operations on matrices. Let $B=(b_{ij})_{m\times m}$ be an arbitrary matrix and $n$ a fixed integer, define
\begin{enumerate}
 \item the column transformation $C_{k,l}^{p}$ of matrix $B$ $(k<l)$: Firstly, columns $k, k+1, \ldots, l-1, l$ are cyclically rotated to columns $l, k, \ldots, l-2, l-1$. Then the rotated column $l$ is added with the column vector $(np,np,\cdots,np)^{T}$.

 \item the row transformation $R_{r,s}^p$ of matrix $B$ $(r>s)$: Firstly, rows $s, s+1, \ldots, r-1, r$ are cyclically rotated to rows $s+1, s+2, \ldots, r, s$. Then the rotated row $s$ is added with the row vector $(np,np,\cdots,np)$.
\end{enumerate}


For instant, suppose $B=(b_{ij})_{4\times 4}$, then
\begin{align*}
C_{1,3}^p(B)=\left(\begin{array}{cccc}
b_{12}&b_{13}&b_{11}+np&b_{14}\\
b_{22}&b_{23}&b_{21}+np&b_{24}\\
b_{32}&b_{33}&b_{31}+np&b_{34}\\
b_{42}&b_{43}&b_{41}+np&b_{44}\\
\end{array}\right), R_{3,1}^p(B)=\left(\begin{array}{cccc}
b_{31}+np&b_{32}+np&b_{33}+np&b_{34}+np\\
b_{11}&b_{12}&b_{13}&b_{14}\\
b_{21}&b_{22}&b_{23}&b_{24}\\
b_{41}&b_{42}&b_{43}&b_{44}\\
\end{array}\right).
\end{align*}
We adapt the short-hand notations $C_{k,l}$ and $R_{r,s}$ for $C_{k,l}^1$ and $R_{r,s}^1$ respectively. We also denote by $\c_d(B)$ (resp. $\r_d(B)$) the $d$th column (resp. row) vector of matrix $B$.

\begin{lem}
Let $\pi\subset\varpi$ be any two partitions and $\varpi/\pi$ is a border strip with $r+1$ rows. Suppose the starting box of $\varpi/\pi$ occupies the $p$th row of $\varpi$ and $\varpi/\pi$ has $a_{p+i}$ boxes in the $(p+i)$th row of $\varpi$ ($i=0,1,\cdots,r$). Then we can write $\varpi/\pi$ as $(\underbrace{0,\cdots,0}\limits_{p-1},\underbrace{a_{p},\cdots,a_{p+r}}\limits_{r+1},\underbrace{0,\cdots,0}\limits_{l(\varpi)-p-r})$. With these notations, we have the following simple facts
\begin{align}\label{e:identity1}
\begin{cases}
a_{p}=\varpi_{p}-\pi_{p}\\
a_{d+1}=\pi_d-\pi_{d+1}+1& ~\text{for $d=p,p+1,\ldots,p+r-1$}.
\end{cases}
\end{align}
and
\begin{align}\label{e:identity3}
\begin{cases}
a_{p+r}=\varpi_{p+r}-\pi_{p+r}\\
a_{d}=\varpi_{d}-\varpi_{d+1}+1& ~\text{for $d=p,p+1,\ldots,p+r-1$}.
\end{cases}
\end{align}
Taking sum on both sides in \eqref{e:identity1} or \eqref{e:identity3} further gives
\begin{align}\label{e:identity2}
\varpi_{p}-\pi_{p+r}+r=|\varpi/\pi|.
\end{align}
\end{lem}

\begin{prop}\label{t:transform} Let $\mu=(\mu_1,\mu_2,\cdots,\mu_l)\subset\lambda=(\lambda_1,\lambda_2,\cdots,\lambda_l)$ be two partitions and
$\mu\subset\nu\subset\lambda.$
\begin{enumerate}
\item If $\nu/\mu=(\underbrace{0,\cdots,0}\limits_{j_1-1},a_{j_1},\cdots,a_{j_s},\underbrace{0,\cdots,0}\limits_{l-j_s})$ is a $\tau_1n$-border strip, then $$\A(\lambda/\mu)=C_{j_1,j_s}^{\tau_1}(\A(\lambda/\nu));$$\\
\item If $\lambda/\nu=(\underbrace{0,\cdots,0}\limits_{i_1-1},b_{i_1},\cdots,b_{i_r},\underbrace{0,\cdots,0}\limits_{l-i_r})$ is a $\tau_2n$-border strip, then $$\A(\lambda/\mu)=R_{i_r,i_1}^{\tau_2}(\A(\nu/\mu)).$$
\end{enumerate}
\end{prop}
\begin{proof}
$(1)$  By definitions
\begin{align}\label{e:nu}
\nu_{d}=
\begin{cases}
\mu_{d}& \text {if $1\leq d\leq j_1-1$};\\
\mu_{d}+a_{d}& \text {if $j_1\leq d\leq j_s$};\\
\mu_{d}& \text {if $j_s+1\leq d\leq l$}.
\end{cases}
\end{align}
 And
\begin{align*}
\c_{d}(\A(\lambda/\mu))=\left(\begin{array}{c}
\lambda_{1}-1-\mu_{d}+d\\
\lambda_{2}-2-\mu_{d}+d\\
\vdots\\
\lambda_{l}-l-\mu_{d}+d\\
\end{array}\right), \c_{d}(\A(\lambda/\nu))=\left(\begin{array}{c}
\lambda_{1}-1-\nu_{d}+d\\
\lambda_{2}-2-\nu_{d}+d\\
\vdots\\
\lambda_{l}-l-\nu_{d}+d\\
\end{array}\right).
\end{align*}
We check $(1)$ case by case.
\begin{itemize}
\item If $1\leq d\leq j_{1}-1$ or $j_s+1\leq d\leq l$, then clearly $\c_{d}(\A(\lambda/\mu))=\c_{d}(\A(\lambda/\nu))$.

\item If $j_1\leq d\leq j_s-1$, then by \eqref{e:nu}
$
\c_{d+1}(\A(\lambda/\nu))=\left(\begin{array}{c}
\lambda_{1}-1-\mu_{d+1}-a_{d+1}+d+1\\
\lambda_{2}-2-\mu_{d+1}-a_{d+1}+d+1\\
\vdots\\
\lambda_{l}-l-\mu_{d+1}-a_{d+1}+d+1\\
\end{array}\right).
$
It follows from \eqref{e:identity1} that
\begin{align*}
a_{d+1}=\mu_{d}-\mu_{d+1}+1, \quad \text {$j_1\leq d\leq j_s-1$}.
\end{align*}
Therefore $\c_{d}(\A(\lambda/\mu))=\c_{d+1}(\A(\lambda/\nu)), ~d=j_{1},\cdots,j_{s}-1$.

\item If $d=j_s$, by \eqref{e:identity2} we then have $-\mu_{j_s}+j_s=\tau_{1}n-\nu_{j_1}+j_1.$
So
$$
\c_{j_{s}}(\A(\lambda/\mu))=\c_{j_{1}}(\A(\lambda/\nu))+(\tau_1n, \cdots, \tau_1n)^T.
$$
\end{itemize}
Summarizing the three cases, we prove that $\A(\lambda/\mu)=C_{j_1,j_s}^{\tau_1}(\A(\lambda/\nu))$.

$(2)$ This can be derived similarly to $(1)$. Just note that the following fact
\begin{align*}
\nu_{d}=\begin{cases}
\lambda_{d}& \text {if $1\leq d\leq i_1-1$}\\
\lambda_{d}-b_{d}& \text {if $i_1\leq d\leq i_r$}\\
\lambda_{d}& \text {if $i_r<d \leq l$},
\end{cases}%
\end{align*}
which together with \eqref{e:identity3} gives $(2)$, completing the proof.
\end{proof}
The following is a straightforward conclusion by repeated applying Prop. \ref{t:transform}.
\begin{cor}\label{t:transform2}
Let
$$
\mu=\alpha^{(0)}\subset\alpha^{(1)}\subset\alpha^{(2)}\subset\cdots\subset\alpha^{(k)}=\lambda
$$
be an $n$-border strip chain of weight $\tau=(\tau_1,\tau_2,\cdots,\tau_k).$
If we denote $$\alpha^{(i)}/\alpha^{(i-1)}=(\underbrace{0,\cdots,0}\limits_{j_i-1},a^{(i)}_{j_i},a^{(i)}_{j_i+1},\cdots,a^{(i)}_{j_i+r_i},
\underbrace{0,\cdots,0}\limits_{l(\la)-j_i-r_i}),\quad i=1,2,\cdots,k.$$ Then
\begin{align*}
\A(\lambda/\mu)=\left(\prod_{1\leq b\leq i-1}C_{j_b,j_b+r_b}^{\tau_b}\right)\left(\prod_{k\geq a\geq i+1}R_{j_a+r_a,j_a}^{\tau_a}\right)\A(\alpha^{(i)}/\alpha^{(i-1)})
\end{align*}
where
\begin{align*}
\prod_{1\leq b\leq i-1}C_{j_b,j_b+r_b}^{\tau_b}&:=C_{j_1,j_1+r_1}^{\tau_1}C_{j_2,j_2+r_2}^{\tau_2}\cdots C_{j_{i-1},j_{i-1}+r_{i-1}}^{\tau_{i-1}};\\ \prod_{k\geq a\geq i+1}R_{j_a+r_a,j_a}^{\tau_a}&:=R_{j_k+r_k,j_k}^{\tau_k}R_{j_{k-1}+r_{k-1},j_{k-1}}^{\tau_{k-1}}\cdots R_{j_{i+1}+r_{i+1},j_{i+1}}^{\tau_{i+1}}.
\end{align*}
\end{cor}

In the following two theorems, we give a combinatorial interpretation for $\det(\M(\la/\mu))$, which demonstrates the equivalence of the determinant-type plethystic Murnaghan--Nakayama rule and the combinatorial one.
\begin{thm}\label{t:horizontal}
Let $\mu=(\mu_1,\mu_2,\cdots,\mu_l)\subset\lambda=(\lambda_1,\lambda_2,\cdots,\lambda_l)$ be two partitions. If $\lambda/\mu$ is a horizontal $n$-border strip of weight $k$, then $\rm{det}(\M(\lambda/\mu))=\sgn(\lambda/\mu).$
\end{thm}
\begin{proof}
We argue by induction on $k.$ The case $k=1$ holds due to the classical Murnaghan--Nakayama rule. Assume it holds for $<k.$ Consider the case where $\lambda/\mu$ is a horizontal $n$-border strip of weight $k$ :
\begin{align*}
\mu=\alpha^{(0)}\subset\alpha^{(1)}\subset\alpha^{(2)}\subset\cdots\subset\alpha^{(k)}=\lambda
\end{align*}
where $\alpha^{(i)}/\alpha^{(i-1)}=(\underbrace{0,\cdots,0}\limits_{j_i-1},a^{(i)}_{j_i},a^{(i)}_{j_i+1},\cdots,a^{(i)}_{j_i+r_i},
\underbrace{0,\cdots,0}\limits_{l-j_i-r_i})$ is an $n$-border strip. In particular, $\alpha^{(1)}/\mu$ is an $n$-border strip.
It follows from \eqref{e:identity2} that
\begin{align}\label{ch2-se2-eq2}
\alpha^{(1)}_{j_{1}}-\mu_{j_{1}+r_{1}}+r_{1}=\sum_{h=j_{1}}^{j_{1}+r_{1}}a^{(1)}_{h}=n.
\end{align}
 By Prop. \ref{t:transform} $(1)$, we have $\A(\lambda/\mu)=C_{j_1,j_1+r_1}(\A(\lambda/\alpha^{(1)})),$ which means
\begin{align*}
\c_{j_{1}+r_{1}}(\A(\lambda/\mu))=\c_{j_{1}}(\A(\lambda/\alpha^{(1)}))+(n, \cdots, n)^T.
\end{align*}
By definition,
\begin{align*}
\c_{j_{1}}(\A(\lambda/\alpha^{(1)}))=\left(\begin{array}{c}
\lambda_{1}-1-\alpha^{(1)}_{j_{1}}+j_{1}\\
\vdots\\
\lambda_{j_{1}+r_{1}}-j_{1}-r_{1}-\alpha^{(1)}_{j_{1}}+j_{1}\\
\lambda_{j_{1}+r_{1}+1}-j_{1}-r_{1}-1-\alpha^{(1)}_{j_{1}}+j_{1}\\
\vdots\\
\lambda_{l}-l-\alpha^{(1)}_{j_{1}}+j_{1}\\
\end{array}\right).
\end{align*}
We claim that all entries in $\c_{j_{1}}(\A(\lambda/\alpha^{(1)}))$ are not equal to $-n$. This is shown by distinguishing two cases whether the entries are above the $ (j_{1}+r_{1})$th row or not.

\begin{enumerate}[label={\bf Case \arabic*.}]
\item  When $1\leq t \leq j_{1}+r_{1}$, then we have
\begin{align*}
\lambda_{t}-t-\alpha^{(1)}_{j_{1}}+j_{1}&\geq\lambda_{j_{1}+r_{1}}-j_{1}-r_{1}-\alpha^{(1)}_{j_{1}}+j_{1}~~~(\text{ by property {\bf (A)} of $A(\lambda/\alpha^{(1)})$ })\\
&\geq \alpha^{(1)}_{j_{1}+r_{1}}-r_{1}-\alpha^{(1)}_{j_{1}}~~~~(\text { by $\lambda_{j_{1}+r_{1}}\geq\alpha^{(1)}_{j_{1}+r_{1}}$})\\
&>\mu_{j_{1}+r_{1}}-r_{1}-\alpha^{(1)}_{j_{1}}=-n~~~~(\text { by $\alpha^{(1)}_{j_{1}+r_{1}}>\mu_{j_{1}+r_{1}}$ and \eqref{ch2-se2-eq2}}).
\end{align*}
 \item When $j_{1}+r_{1}+1\leq t\leq l$, by Remark \ref{R:chacter} $(1)$, the ending box of $\alpha^{(1)}/\mu$ is the southwesternmost box in $\la/\mu$, which means $\lambda_{j_{1}+r_{1}+1}=\mu_{j_{1}+r_{1}+1}\leq\mu_{j_{1}+r_{1}}$.
\begin{align*}
\lambda_{t}-t-\alpha^{(1)}_{j_{1}}+j_{1} &\leq \lambda_{j_{1}+r_{1}+1}-j_{1}-r_{1}-1-\alpha^{(1)}_{j_{1}}+j_{1}~~~(\text{ by property {\bf (A)} of $A(\lambda/\alpha^{(1)})$})\\
&\leq\mu_{j_{1}+r_{1}}-r_{1}-1-\alpha^{(1)}_{j_1}~~~(\text { by $\lambda_{j_{1}+r_{1}+1}=\mu_{j_{1}+r_{1}+1}\leq\mu_{j_{1}+r_{1}}$})\\
&=-(\alpha^{(1)}_{j_1}-\mu_{j_{1}+r_{1}}+r_{1})-1=-n-1<-n~~~(\text {by \eqref{ch2-se2-eq2}}).
\end{align*}
\end{enumerate}
This together with $\A(\lambda/\mu)=C_{j_1,j_1+r_1}(\A(\lambda/\alpha^{(1)}))$ implies that $\M(\lambda/\mu)$ can be obtained by exchanging the $j_1$th column in $\M(\lambda/\alpha^{(1)})$ $r_1$ times to the right. Thus $\det(\M(\lambda/\mu))=(-1)^{r_1}\det(\M(\lambda/\alpha^{(1)}))$. By the inductive hypothesis, we have $$\det(\M(\lambda/\mu))=(-1)^{r_1}\det(\M(\lambda/\alpha^{(1)}))=(-1)^{r_1}\sgn(\lambda/\alpha^{(1)})=\sgn(\lambda/\mu),$$
completing the proof.
\end{proof}


Let $\lambda=(\lambda_1, \ldots, \lambda_l)\supset \mu=(\mu_1, \ldots, \mu_l)$ be two partitions. If $\D(\la/\mu)\neq\varnothing$, by Remark \ref{R:chacter}, we can find a longest horizontal $n$-border strip chain $\underline{\beta}$ (denote its weight by $\tau=(\tau_1,\tau_2,\cdots,\tau_m)$):
\begin{align*}
\mu=\beta^{(0)}\subset\beta^{(1)}\subset\cdots\subset\beta^{(m)}=\lambda.
\end{align*}
Denote by $\N(\underline{\beta})$ the set of $1\leq i\leq m$ such that $\tau_i>1$ and the ending box of $\beta^{(i)}/\beta^{(i-1)}$ occupies the bottom row of $\la/\mu$.

\begin{lem}\label{L:zero}
With the above notations, if $\N(\underline{\beta})\neq\varnothing$, then $\det(\M(\la/\mu))=0$.
\end{lem}
\begin{proof}
Let us say the bottom row of $\la/\mu$ is numbered $r$. Then $\A(\la/\mu)$ can be written as
\begin{align*}
 \A(\lambda/\mu)=\left(\begin{array}{cc}
\A_{1}(\lambda/\mu) & \ast\\
<0 & \A_{2}(\lambda/\mu)
\end{array}\right)\Longrightarrow \M(\lambda/\mu)=\left(\begin{array}{cc}
\M_{1}(\lambda/\mu) & \ast\\
0 & \M_{2}(\lambda/\mu)
\end{array}\right)
\end{align*}
where $\A_{1}(\lambda/\mu)$ is the $r\times r$ submatrix, $\A_{2}(\lambda/\mu)$ is a submatrix with principal diagonal zero and lower triangular negative. So
$\M_{2}(\lambda/\mu)$ is an upper triangular matrix with principal diagonal of $1$, and thus $\det(\M(\lambda/\mu))=\det(\M_{1}(\lambda/\mu)$). Without loss of generality, we can assume the bottom rows of $\la/\mu$ and $\la$ coincide, i.e., $r=l(\la)$.

Choose $j\in\N(\underline{\beta})$. It follows from Remark \ref{R:chacter} $(2)$ that $\tau_j>1$ and $\beta^{(j)}/\beta^{(j-1)}$ consists of $\tau_j$ disjoint $n$-border strips such that the starting box of the lower border strip is directly below the ending box of the upper border strip (denote these $n$-border strips by $\Theta_1,\Theta_2,\cdots,\Theta_{\tau_j}$ from the top down). We further assume that the starting box of $\Theta_{\tau_j}$ is in the $(q+1)$th row of $\la/\mu$. Since the ending box of $\Theta_{\tau_j-1}$ is directly above the starting box of $\Theta_{\tau_j}$,
\begin{align}\label{e:I1}
\beta^{(j)}_{q+1}=\beta^{(j-1)}_{q}+1.
\end{align}
By \eqref{e:identity2},
\begin{align}\label{e:I2}
\beta^{(j)}_{q+1}-\beta^{(j-1)}_{r}+r-q-1=|\Theta_{\tau_j}|=n.
\end{align}
Combining \eqref{e:I1} with \eqref{e:I2}, we have
\begin{align}\label{e:I3}
\beta^{(j-1)}_q-\beta^{(j-1)}_{r}+r-q=n.
\end{align}
By definition, we have
\begin{align}\label{e:I4}
\begin{split}
&\A(\beta^{(j)}/\beta^{(j-1)})_{rr}-\A(\beta^{(j)}/\beta^{(j-1)})_{rq}\\
=&\beta^{(j)}_{r}-\beta^{(j-1)}_{r}-(\beta^{(j)}_{r}-\beta^{(j-1)}_{q}-r+q)\\
=&\beta^{(j-1)}_q-\beta^{(j-1)}_{r}+r-q=n \quad\text{(by \eqref{e:I3})}.
\end{split}
\end{align}
It follows from property {\bf B} of $\A(\beta^{(j)}/\beta^{(j-1)})$ that
\begin{align}\label{e:I5}
\c_{r}(\A(\beta^{(j)}/\beta^{(j-1)}))-\c_{q}(\A(\beta^{(j)}/\beta^{(j-1)}))=(n,n,\cdots,n)^{T}.
\end{align}
The choice of $j$ forces that there exists at least one box of $\beta^{(j)}/\beta^{(j-1)}$ in the bottom row of $\la/\mu$, which implies
\begin{align}
\A(\beta^{(j)}/\beta^{(j-1)})_{rr}=\beta^{(j)}_{r}-\beta^{(j-1)}_{r}\geq1.
\end{align}
Together with \eqref{e:I4} this yields that
\begin{align}\label{e:I6}
\A(\beta^{(j)}/\beta^{(j-1)})_{rq}=\A(\beta^{(j)}/\beta^{(j-1)})_{rr}-n\geq1-n>-n.
\end{align}
By combining Eq. \eqref{e:I6}, property {\bf A} of $\A(\beta^{(j)}/\beta^{(j-1)})$ and Eq. \eqref{e:I5}, we conclude that
all entries in $q$th column and $r$th column of $\A(\beta^{(j)}/\beta^{(j-1)})$ are greater than $-n$ and the difference between these two columns is $n$. By Cor. \ref{t:transform2}, $\A(\lambda/\mu)$ can
be obtained by a series of the row transformations and the column transformations of $\A(\beta^{(j)}/\beta^{(j-1)})$, which means there exists two columns in $\A(\lambda/\mu)$ such that the elements in these two columns are all greater than $-n$ and their difference is a multiple of $n$. Thus there are always two identical columns in $\M(\lambda/\mu)$, this derives $\det(\M(\lambda/\mu))=0.$
\end{proof}

\begin{thm}\label{t:nonhorizontal}
 Let $\mu\subset\lambda$ be two partitions with $|\la/\mu|=kn$. If $\lambda/\mu$ is not a horizontal $n$-border strip of weight $k$, then $\det(\M(\lambda/\mu))=0.$
\end{thm}
\begin{proof}
We check this by induction on $k$. The initial step can be derived by the classical Murnaghan-Nakayama rule. Assume this holds for all cases $k^{'}<k$. Now we consider the case $k^{'}=k$. The statement at the start of this subsection tells us it suffices to consider the case $\D(\la/\mu)\neq\varnothing$. In this case, we can find a longest horizontal $n$-border strip chain $\underline{\beta}$ of weight $\tau=(\tau_1,\tau_2,\cdots,\tau_m)$:
\begin{align*}
\mu=\beta^{(0)}\subset\beta^{(1)}\subset\cdots\subset\beta^{(m)}=\lambda.
\end{align*}
such that there exists at least one $\tau_i>1$. If we denote $\Phi_i=\beta^{(i)}/\beta^{(i-1)}$ ($1\leq i\leq m$), then,
the starting box of $\Phi_i$ is southwest of that of $\Phi_j$ if and only if $i<j$. By Lemma \ref{L:zero}, we have done if $\N(\underline{\beta})\neq\varnothing$. So we assume $\N(\underline{\beta})=\varnothing$. In this case, we claim that the ending box of $\Phi_1$ is exactly the southwesternmost box of $\la/\mu$. Indeed, if not, we can find a $\Phi_{c}$ such that the ending box (denote by $x$) of $\Phi_{c}$ is on the bottom row of $\la/\mu$ and the content of $x$ is strictly less than that of the starting box of $\Phi_{1}$, which forces $|\Phi_{c}|>n$ (equivalently $\tau_{c}>1$). This contradicts with $\N(\underline{\beta})=\varnothing$. We further have $|\Phi_1|=n$ by $\N(\underline{\beta})=\varnothing$. Similarly to the proof of Theorem \ref{t:horizontal}, we have $\det(\M(\la/\mu))=(-1)^{ht(\Phi_1)}\det(\M(\la/\beta^{(1)}))$. Note that $\underline{\beta}\setminus\Phi_1$ is a longest horizontal $n$-border strip chain of weight $\tau^{[1]}=(\tau_2,\cdots,\tau_m)$ for $\la/\beta^{(1)}$. Since $\tau_1=1$, there must exist a $\tau_i>1$ ($2\leq i\leq k$), which means $\la/\beta^{(1)}$ is not a horizontal $n$-border strip of weight $k-1$ by Prop. \ref{p:characterization}. By the induction hypothesis, $\det(\M(\la/\mu))=(-1)^{ht(\Phi_1)}\det(\M(\la/\beta^{(1)}))=0$. This completes the proof.
\end{proof}

The following result is immediate by combining Thm. \ref{t:horizontal} and Thm. \ref{t:nonhorizontal}.
\begin{cor}\label{t:dpM-N} The determinant-type plethystic Murnaghan-Nakayama rule can be written as
\begin{align}
V^{(n)}_{k}s_{\mu}=(p_{n}\circ h_{k})s_{\mu}=\sum_{\lambda\vdash |\mu|+nk}\det(\M(\lambda/\mu))s_{\lambda}
\end{align}
where $\det(\M(\lambda/\mu))$ admits the following combinatorial interpretation:
\begin{align}
\det(\M(\lambda/\mu))=
\begin{cases}
\sgn(\la/\mu) & \text{if $\lambda/\mu$ is a horizontal $n$-border strip of weight $k$};\\
0 & \text{otherwise}.
\end{cases}
\end{align}
This means the combinatorial plethystic Murnaghan-Nakayama rule (\ref{t:c}) is equivalent to our determinant-type one.
\end{cor}


\subsection{Generalized Waring Formula}

Waring's formula expresses the power sum symmetric functions in terms of the elementary
symmetric functions. Merca \cite{M} proposed a generalization of Waring's formula
as follows.

Let $k, n$ and $r$ be three positive integers and let $x_{1}, x_{2}, \ldots, x_{n}$ be independent variables. Then
\begin{align}
e_{k}\left(x_{1}^{r}, x_{2}^{r}, \ldots, x_{n}^{r}\right)=(-1)^{k(r+1)} \sum_{\substack{\lambda \vdash k r \\ \ell(\lambda) \leq r}} \frac{M_{\lambda}^{(k, r)}}{\prod_{i=1}^{k r} m_{i}(\lambda) !} e_{\lambda}\left(x_{1}, x_{2}, \ldots, x_{n}\right),
\end{align}
where
$$
\begin{aligned}
M_{\left(\lambda_{1}, \lambda_{2}, \ldots, \lambda_{\ell}\right)}^{(k, r)}=&-\sum_{i=1}^{\ell-1} M_{\left(\lambda_{1}, \ldots, \lambda_{i-1}, \lambda_{i}+\lambda_{\ell}, \lambda_{i+1}, \ldots, \lambda_{\ell-1}\right)}^{(k, r)} \\
&+\left\{\begin{array}{ll}
r \cdot M_{\left(\lambda_{1}, \lambda_{2}, \ldots, \lambda_{\ell-1}\right)}^{\left(k-\lambda_{\ell} / r, r\right)}, & \text { if } \lambda_{\ell} \equiv 0 \quad(\bmod~ r), \\
0, & \text { otherwise },
\end{array}\right.
\end{aligned}
$$
with the initial condition $M_{\left(\lambda_{1}\right)}^{(k, r)}=r$.

Recall that the involution $\omega: \Lambda_{\mathbb{Q}} \rightarrow\Lambda_{\mathbb{Q}}$ \cite{Mac} is an isometry with respect the canonical inner product such that
\begin{align*}
\omega(e_{\lambda})=h_{\lambda}, \quad \omega(s_{\lambda})=s_{\lambda^{'}},\quad \omega(p_{\lambda})=(-1)^{|\lambda|-\ell(\lambda)}p_{\lambda} .
\end{align*}
So we have
\begin{align}
\omega(p_n \circ h_k)&=(-1)^{k(n-1)}p_n \circ e_k.
\end{align}
We now discuss a simpler method to compute the generalized
Waring coefficients as follows.
\begin{cor}\label{c:Waring} (Generalized Waring Formula) Let $n$ and $k$ be two positive integers, then
\begin{align}
p_n \circ e_k=(-1)^{k(n-1)}\sum_{\substack{\lambda \vdash n k \\ \ell(\lambda) \leq n}}\det(\M(\lambda))\det(e_{\lambda_{i}+j-i})_{i,j=1}^{\ell(\lambda)}
\end{align}
summed over all horizontal $n$-border strips of weight $k$.
\begin{proof}
It follows from Cor. \ref{t:dpM-N} and the Jacobi-Trudi formula.
\end{proof}
\end{cor}

\section*{Acknowledgments}

The project is partially supported by Simons Foundation under
grant no. 523868 and NSFC grant 12171303.

\bigskip
\noindent{\bf Data Availability Statement}

All data generated during the study are included in the article.

\bigskip

\bibliographystyle{plain}

\end{document}